\pdfoutput=1
\documentclass[11pt,reqno]{amsart}
\usepackage[letterpaper,margin=1in,footskip=0.25in]{geometry}
\usepackage{mathrsfs}
\usepackage{amssymb}
\usepackage{mathtools}
\usepackage{tikz-cd}
\usepackage{enumitem}
\usepackage[all,cmtip]{xy}

\PassOptionsToPackage{pdfusetitle,pagebackref,colorlinks}{hyperref}
\usepackage{bookmark}
\hypersetup{
  linkcolor={red!70!black},
  citecolor={green!70!black},
  urlcolor={blue!80!black}
}
\usepackage[capitalize]{cleveref}
\makeatletter
  \def\th@plain{
  \thm@headfont{\bfseries} 
  \thm@notefont{\itshape} 
  \itshape
}
\makeatother

\makeatletter
  \def\th@definition{
  \thm@headfont{\bfseries} 
  \thm@notefont{\bfseries} 
}
\makeatother

\makeatletter
  \def\th@remark{
  \thm@headfont{\bfseries} 
  \thm@notefont{\bfseries} 
	}
\makeatother

\newtheorem{theorem}{Theorem}[section]
\newtheorem{lemma}[theorem]{Lemma}
\newtheorem{proposition}[theorem]{Proposition}
\newtheorem{corollary}[theorem]{Corollary}

\newtheorem{conjecture}[theorem]{Conjecture}

\theoremstyle{question}
\newtheorem{question}[theorem]{Question}
\theoremstyle{definition}
\newtheorem{definition}[theorem]{Definition}

\theoremstyle{remark}
\newtheorem{remark}[theorem]{Remark}

\newtheoremstyle{cited}{.5\baselineskip\@plus.2\baselineskip\@minus.2\baselineskip}{.5\baselineskip\@plus.2\baselineskip\@minus.2\baselineskip}{\itshape}{}{\bfseries}{\bfseries .}{5pt plus 1pt minus 1pt}{\thmname{#1}\thmnumber{~#2}\thmnote{ \normalfont#3}}
\theoremstyle{cited}
\newtheorem{citedthm}[theorem]{Theorem}

\newtheoremstyle{citeddef}{.5\baselineskip\@plus.2\baselineskip\@minus.2\baselineskip}{.5\baselineskip\@plus.2\baselineskip\@minus.2\baselineskip}{}{}{\bfseries}{\bfseries .}{5pt plus 1pt minus 1pt}{\thmname{#1}\thmnumber{~#2}\thmnote{ \normalfont#3}}
\theoremstyle{citeddef}

\newcommand{\C}{\mathbb{C}}
\newcommand{\Z}{\mathbb{Z}}
\newcommand{\Hom}{\mathrm{Hom}}

\newcommand{\K}{\mathcal{L}}

\def\be{\begin{equation}}
\def\ee{\end{equation}}

\def\bt{\begin{theorem}}
\def\et{\end{theorem}}

\def\bc{\begin{corollary}}
\def\ec{\end{corollary}}

\def\br{\begin{remark}}
\def\er{\end{remark}}

\def\bp{\begin{proposition}}
\def\ep{\end{proposition}}

\def\bl{\begin{lemma}}
\def\el{\end{lemma}}

\def\bex{\begin{ex}}
\def\eex{\end{ex}}

\def\bd{\begin{definition}}
\def\ed{\end{definition}}

\DeclareMathOperator{\codim}{codim}              
\DeclareMathOperator{\reg}{reg}                  
                  
\DeclareMathOperator{\id}{id}                    

\DeclareMathOperator{\dol}{dol}

\DeclareMathOperator{\Char}{Char}
\DeclareMathOperator{\CC}{CC}

\DeclareMathOperator{\im}{Im}

\def\Pic{{\rm Pic}}

\def\bZ{\mathbb{Z}}

\usepackage{tikz-cd}
\usepackage{enumitem}

\PassOptionsToPackage{pdfusetitle,pagebackref,colorlinks}{hyperref}
\usepackage{bookmark}
\hypersetup{
  linkcolor={red!70!black},
  citecolor={green!70!black},
  urlcolor={blue!80!black}
}

\newcommand{\sF}{\mathcal{F}}

\newcommand{\sH}{\mathcal{H}}
\newcommand{\sK}{\mathcal{K}}
\newcommand{\sL}{\mathcal{L}}

\newcommand{\sO}{\mathcal{O}}
\newcommand{\sP}{\mathcal{P}}

\newcommand{\sR}{\mathcal{R}}
\newcommand{\sT}{\mathcal{T}}
\newcommand\sV{{\mathcal V}}

\newcommand{\bbC}{\mathbb{C}}

\newcommand{\bbK}{\mathbb{K}}
\newcommand{\bbL}{\mathbb{L}}

\newcommand{\bbP}{\mathbb{P}}

\newcommand{\bbR}{\mathbb{R}}

\newcommand{\bbZ}{\mathbb{Z}}

\newcommand{\scrC}{\mathscr{C}}

\newcommand{\into}{\hookrightarrow}
\newcommand{\onto}{\rightarrow\hspace*{-.14in}\rightarrow}

\renewcommand{\SS}{\operatorname{SS}}

\DeclareMathOperator{\torsion}{torsion}

\newcommand{\TC}{\mathrm{TC}}

\begin{document}  
\title[Generic Vanishing, 1-forms, and Topology of Albanese Maps]{Generic Vanishing, 1-forms, and Topology of Albanese Maps} 

\author{Yajnaseni Dutta}

\address{Mathematisch Instituut, Universiteit Leiden, Niels Bohrweg 1, 2333CA, Leiden, NL.}
\email{y.dutta@math.leidenuniv.nl}

\author{Feng Hao}

\address{School of Mathematics, Shandong University, 27 Shanda South Road, Jinan 250100 China.}
\email{feng.hao@sdu.edu.cn}

\author{Yongqiang Liu}

\address{The Institute of Geometry and Physics, University of Science and Technology of China, 96 Jinzhai Road, Hefei 230026 China.}
\email{liuyq@ustc.edu.cn}

\subjclass[2020]{primary 32Q55, 32S60, 14K12,  14C30 secondary  14F40} 
\keywords{Holomorphic one-form, cohomology jump loci, constructible sheaves, singular support, linearity, Abelian variety}

\begin{abstract}

 Given a bounded constructible complex of sheaves $\sF$ on a complex Abelian variety, we prove an equality relating the cohomology jump loci of $\sF$ and its singular support. 
As an application, we identify two subsets of the set of holomorphic 1-forms with zeros on a complex smooth projective irregular variety $X$; one from Green-Lazarsfeld's cohomology jump loci and one from the Kashiwara's estimates for singular supports. This result is related to Kotschick's conjecture about the equivalence between the existence of nowhere vanishing global holomorphic 1-forms and the existence of a fibre bundle structure over the circle. Our results give a conjecturally equivalent formulation using singular support, which is equivalent to a criterion involving cohomology jump loci proposed by Schreieder. As another application, we reprove a recent result proved by Schreieder and Yang; namely if $X$ has simple Albanese variety and admits a fibre bundle structure over the circle, then the Albanese morphism cohomologically behaves like a smooth morphism with respect to integer coefficients.

In a related direction, we address the question whether the set of 1-forms that vanish somewhere is a finite union of linear subspaces of $H^0(X,\Omega_X^1)$. We show that this is indeed the case for forms admitting zero locus of codimension 1. 
\end{abstract}

\maketitle
\section{Introduction}\label{intro}

We study the topology of the Albanese map via constructible complexes on Abelian varieties. The latter has been recently explored extensively by Schnell \cite{Sch15} using generalised Fourier--Mukai tranforms and the language of holonomic D-modules. His results vastly extended the foundational results in \cite{GL87,GL91, sim93b, Ara97} about cohomology jump loci of rank one local systems, as well as their incarnations in the moduli of line bundles with connections or rank 1 higgs bundles on smooth irregular projective varieties. 
In particular, Schnell proved a structure theorem for cohomology jump loci for any bounded constructible complex of sheaves on complex Abelian varieties. 
To state Schnell's results, we first recall the definition of cohomology jump loci. 

Let $D^b_c(A)$ denote the category of bounded constructible complex of $\C$-sheaves on a complex Abelian variety $A$. Set  $\Char^0(A)\coloneqq \Hom(H_1(A, \bbZ), \bbC^{\ast})$, which is the moduli space of rank one $\C$-local systems on $A$.
\begin{definition} 
The $i$-th\textsl{ cohomology jump loci} of $\sF\in D^b_c(A)$  is defined to be the set
\[\sV^i(A, \sF) \coloneqq \{\rho\in\Char^0(A)| H^i(A, \sF\otimes \bbC_{\rho}) \neq 0\},\]
where $\bbC_{\rho}$ denotes the rank 1 local system associated to $\rho\in \Char^0(A)$.
Furthermore, we write $\displaystyle\sV(A,\sF)\coloneqq\bigcup_i \sV^i(A,\sF)$.
\label{def:cjl}
\end{definition}
Schnell's structure theorem states that $ \sV^i(A,\sF)$ is finite union of translated subtori  of $\Char^0(A)$ \cite{Sch15}. 
 In this paper, we prove an equality relating the cohomology jump loci $\sV(A,\sF)$ and the singular support of $\sF$ on complex Abelian varieties. Moreover, we show that the projection of  the singular support of $\sF$ is linear, which is compatible with Schnell's structure theorem on cohomology jump loci, see Remark \ref{rm schnell}.

\bigskip

 The $(1,0)$-piece of the tangent space $\TC_{\rho}(\Char^0(A))$ at a character $\rho \in \Char^0(A)$ is $H^0(A,\Omega_A^1)$. 
Let $\sT(A, \sF)$ denote the union of the $(1,0)$-part of the tangent space to the irreducible components of the subvariety $\sV(A, \sF)$.
\bt
Let $A$ be a complex Abelian variety. For any $\sF \in D^b_c(A)$,  we have the equality
\be \label{equality}
\pi(\SS(\sF)) =  \sT(A, \sF),
\ee
where
$\SS(\cdot)$ denotes the singular support (see Definition \ref{singular support for complex}) of constructible complexes in the cotangent space $T^*A\simeq A\times H^0(A, \Omega_A^1)$  and $\pi\colon T^*A\to H^0(A, \Omega_A^1)$ is the natural projection.
In particular, $\pi(\SS(\sF))$ is a finite union of linear subspaces of the vector space $H^0(A, \Omega_A^1)$.
\label{thm:perverse}
\et
\br \label{rm schnell} The linearity part of the theorem could also follow from the structure theorem for $\sV(A,\sF)$ proved by Schnell \cite[Theorem 7.3]{Sch15}, once the equality (\ref{equality}) is provided.
 Here we give a direct proof of the linearity property of $\pi(\SS(\sF))$ without using Schnell's results (see \cref{van-nonsimple}). 
\er
\br
Due to Riemann-Hilbert correspondence,  \cref{thm:perverse} also holds for regular holonomic $D$-module complexes. 
\er 

The key technique we
use is to relate the two sides of the equality in \cref{thm:perverse} via the
Euler characteristic formula given by 
the Kashiwara index theorem (see \cref{thm:index}). In fact, \cref{thm:perverse} should be viewed as a modified version of Kashiwara index theorem on complex Abelian varieties, since $\SS(\sF)$ records a piece of information about characteristic cycles of $\sF$ and $\sV(A,\sF)$ records a piece of information about the Euler characteristic number of $\sF$.

\medskip
As an application of Theorem \ref{thm:perverse},  
we have the following result proved recently by Schreieder and Yang \cite[Corollary 3.4]{SY22}.
 This prompted us to prove Theorem \ref{thm:perverse} in arbitrary characteristics for simple abelian varieties (see Proposition \ref{any field}) and obtain their result as a corollary. We thank the referee for encouraging us to generalize our main theorem in this direction. 
\begin{corollary} \label{cor general type} Let $f\colon X \to A$ be a morphism from a smooth complex projective variety to a simple Abelian variety.  If there exists a $\scrC^{\infty}$-fibre bundle structure $p_X\colon X\to S^1$ such that $p_X^* (d\theta) \in f^*H^1(A, \bbR)$, where $\theta$ is a
coordinate of the circle,
 then $f$ is a $\Z$-homology fiber bundle. Moreover, for any algebraically closed field coefficient $\bbK$, $\pi(\SS(\bbR f_*\bbK_X))=\{0\}$.
\end{corollary}

\begin{remark}
When the assumption on the simplicity of $A$ is dropped in the Corollary above, our theorem more generally gives information on the topological structure of $\SS(\bbR f_*\bbC_X)$. 
See \cref{corgentype-nonsimple} for more details.
\end{remark}

As another application of Theorem \ref{thm:perverse}, we have the following

\bt Let $X$ be a smooth projective variety with $a\colon X\to A_X$ its Albanese morphism. Under the linear isomorphism $H^0(X,\Omega^1_X)\cong H^0(A_X,\Omega^1_{A_X}) $ one can identify
\be \label{etale equality}
\bigcup_\tau \pi(\SS(\bbR (a\circ \tau)_*\bbC_{X'}))=  \bigcup_\tau \{\omega \in H^0(X,\Omega^1_X)| 
(H^{\bullet}(X', \bbC), \wedge\tau^*\omega) \text{ is not exact}\}, 
\ee
where both unions are running over 
all possible finite \'etale cover $\tau\colon  X' \to X$. 
\label{thm:etale linearity}
\et

The study of the above theorem is motivated by the following conjecture, which was posed by Kotschick \cite[Question 15]{Kot21} and Schreieder \cite{Sch19}.

\begin{conjecture} \label{conjecture}
Let $X$ be a smooth complex projective variety. Then the following three statements are equivalent: 
\begin{enumerate}
	\item\label{kot1} $X$ admits a global holomorphic 1-form
	without zeros.  
	\item\label{kot2} $X$ admits a $\scrC^{\infty}$ real closed 1-form which has no zeros,
	or equivalently $X$ admits a $\scrC^{\infty}$-fibre bundle structure over the circle \cite[Theorem 1]{Tis70}.
	\item\label{kot3} There exists $\omega \in H^0(X,\Omega_X^1)$ such that for all finite \'etale morphism $\tau\colon  X' \to X$, the complex
 $(H^{\bullet}(X', \bbC), \wedge\tau^*\omega)$ is exact. 
\end{enumerate}
\end{conjecture}

Theorem \ref{thm:etale linearity} then gives the fourth criterion, which is equivalent to (\ref{kot3}) in Conjecture \ref{conjecture}:
\begin{enumerate}
\setcounter{enumi}{3} 
	\item\label{kot4}There exists $\omega\in H^0(X,\Omega^1_X) \setminus \bigcup_\tau \pi(\SS(\bbR (a\circ \tau)_*\bbC_{X'}))$.
\end{enumerate}

\begin{remark} \label{resonance-1-form with zero}
All the 1-forms involved in Theorem \ref{thm:etale linearity} are contained in $W(X)$, the collection of holomorphic one forms on $X$ with zeros. 
By a result of Green and Lazarsfeld \cite{GL91} we know that the set on the right side of (\ref{etale equality}) is contained in $W(X)$.
On the other hand, it follows from Kashiwara's estimate (see \cref{thm:kas-ss}) that the left side is also contained in $W(X)$.
Hence  we pose the following question: Is it true that
\be
\overline{ \bigcup_\tau \pi(\SS(\bbR (a\circ \tau)_*\bbC_{X'}))} = \overline{ \bigcup_\bbL \pi(\SS(\bbR a _*\bbL))} = W(X)?
\ee 
Here $\overline{\cdot}$ denotes the Zariski closure and the second union is running over all possible local systems $\bbL$ on $X$.
 
The answer is yes for varieties of dimension less than or equal to 3. In this situation, in fact, using Theorem \ref{thm:etale linearity}, \cite[Corollary 3.1]{Sch19} and \cite[Theorem 1.4]{HS19} we have  $$ \bigcup_\nu \pi(\SS(\bbR (a \circ \nu)_* \C_{X'}))=W(X),$$
with the union running over all possible finite Abelian \'etale covers $\nu\colon X'\to X$.
\er

\subsection{linearity of 1-forms admitting zeros}
By Theorem \ref{thm:perverse}, $\pi(\SS(\bbR (a\circ \tau)_*\bbC_{X'})) $
is a  linear subspace of the vector space $ H^0(X, \Omega_{X}^1)$.  From the point of view of Theorem \ref{thm:etale linearity} and Remark \ref{resonance-1-form with zero}, one may wonder whether the set $W(X)$
is also a finite union of  vector subspaces of $H^0(X,\Omega_X^1)$.  Such a statement is true for the set of global holomorphic tangent vector fields with zeros due to the work of Carrell and Lieberman \cite{CL73}.

 More specifically,  consider 
\[W^i(X) = \{\omega \in H^0(X,\Omega_X^1) \mid  \codim_X Z(\omega) \leq  i\},\]
where $Z(\omega)$ is the zero set of $\omega$. 
 Green and Lazarsfeld showed \cite{GL87} that $W^i(X)$ contains the $(1,0)$-piece of the tangent cone of the cohomology jump loci of $X$ up to degree $i$. Note that the cohomology jump loci of $X$ are finite union of torsion translated sub-tori.
We ask the following
\begin{question}
Are $W^i(X)$ linear for every degree $i$, i.e.\ a finite
union of vector subspaces of the vector space $H^0(X, \Omega_X^1)$?
\label{qn:linearity}
\end{question}

\noindent We answer the question positively for $W^1(X)$. 
\begin{citedthm}[(see \cref{Thm:Proj-codim1})]
 Let $X$ be a smooth projective variety of dimension $n$. Then $W^1(X)$ is linear.
\label{thm:codim1}
\end{citedthm}
 This follows from
a result of Spurr \cite[Theorem 2]{Spu88}; whenever a 1-form $\omega$ vanish along a divisor $E$, one has either $E$ is rigid in the sense that $E^2\cdot H^{n-2}< 0$ with respect to some polarisation $H$ on $X$, or $\omega$ comes from a curve. We generalise this statement
in the setting of a pair (see \cref{thm:spurr}) and prove the linearity statement for logarithmic 1-forms admitting codimension one zeros.

\begin{citedthm}[(see \cref{thm:codim1quasiprojective})]
Let $(X, D)$ be a pair with $X$ a smooth projective variety and $D$ a simple normal crossing divisor of $X$. Then the following set is linear $$ W^1(X,D)\coloneqq\{ \omega \in  H^0(X, \Omega_X^1(\log D)) \mid \codim_X Z(\omega) \leq 1 \}.$$
\end{citedthm}

\subsection*{Convention}  In this paper, all complex of sheaves and perverse sheaves are defined with complex
coefficients except in \S \ref{sec:perverse} and \ref{sec:SY22}. All the varieties are complex quasi-projective varieties. Unless specified otherwise by 1-forms on a smooth projective variety, we mean global holomorphic 1-forms.

\section{Preliminaries}
\subsection{1-forms and associated local systems}\label{sec:gv}
The results  in this subsection should be well known to the experts and we include it here due to the lack of references.

Let X be a smooth projective variety. Consider the identity component of the character variety $\Char(X)\coloneqq \Hom(H_1(X, \bbZ), \bbC^{\ast})$  given by
\[\Char^0(X)\coloneqq \Hom(H_1(X, \bbZ)/\torsion, \bbC^{\ast}).\] 
The $i$-th\textsl{ cohomology jump loci} $\sV^i(X,\sF)\subset \Char^0(X)$ for
$\sF\in D^b_c(X)$  is defined in a similar way as in \cref{def:cjl}. As in the introduction $ \sV(X,\sF)= \bigcup_i \sV^i(X,\sF)$.   The corresponding tangent cone $\sT(X, \sF)\subset H^0(X,\Omega_X^1)$ is also defined similarly as was done before Theorem \ref{thm:perverse}. 
More precisely,
\[\sT(X, \sF)\coloneqq H^0(X, \Omega_{X}^1)\cap \big(  \bigcup_{\rho}   \TC_{\rho} \sV(X, \sF)\big),\] 
where the union is running over representative points from irreducible components of $\sV(X,\sF)$
and
$\TC_{\rho} \sV(X,\sF) \subseteq H^1(X, \bbC)$ denotes the tangent cone at $\rho$. 
When $\sF=\bbC_X$, we simply write $\sT(X)\coloneqq\sT(X, \C_X)$.

\medskip

Given a 1-form $\omega\in H^0(X,\Omega_X^1)$,
the kernel $L(\omega)$ of $\sO_X\overset{d+\wedge\omega}{\longrightarrow} \Omega_X^1$ is a rank 1 local system resolved by the de Rham complex (see the proof of \cite[Lemma 2.2]{Sch19})
\begin{equation}
	\sK^{\bullet}(\omega) \coloneqq [\sO_X\overset{d+\wedge\omega}{\longrightarrow} \Omega_X^1 \longrightarrow \cdots\overset{d+\wedge\omega}{\longrightarrow}  \Omega_X^{n-1}\overset{d+\wedge\omega}{\longrightarrow}  \Omega^n_X],
	\label{eq:koszul}
\end{equation}
and hence $H^k(X, L(\omega))=\mathbb{H}^k(X,(\Omega_X^{\bullet}, d+\wedge \omega))$. What's more, the corresponding line bundle $\sL_{\omega}\coloneqq L(\omega)\otimes_{\C}\sO_X$ is isomorphic to $\sO_X$.  
Hence we have the following short exact sequence 

\begin{equation}
	0\to H^0(X, \Omega_X^1)\to \Char^0(X)\to \Pic^0(X)\to 0,
	\label{eq:lsseq}
\end{equation} 
$$\omega\mapsto L(\omega); L \mapsto L\otimes_{\C}\sO_X,$$

In order to obtain a Kodaira--Nakano-type vanishing theorem
for degree zero line bundles Green--Lazarsfeld \cite{GL87} considered
the following spectral sequence associated to the complex (\ref{eq:koszul}) $$E_1^{p, q}(\omega)\coloneqq H^p(X, \Omega_X^q)\Rightarrow \mathbb{H}^{p+q}(X,(\Omega_X^{\bullet}, d+\wedge \omega)),$$ with differential $d_1=\wedge\omega: E_1^{p, q}(\omega)\to E_1^{p, q+1}(\omega)$ induced by $d+\wedge\omega$ in complex (\ref{eq:koszul}). Using this spectral sequence, \cite[Proposition 3.4]{GL87} shows that if there is a holomorphic 1-form $\omega$ whose zero locus $Z(\omega)$ has codimension $\geq k$, then the sequence
\[\cdots \longrightarrow H^p(X, \Omega^{q-1})
\overset{\wedge\omega}{\longrightarrow}H^p(X, \Omega_X^{q})
\overset{\wedge\omega}{\longrightarrow} H^p(X,\Omega_X^{q+1})\longrightarrow\cdots\]
is exact for all $p+q<k$. Putting these together 
by the Hodge decomposition for $H^i(X,\bbC)$ we get
\begin{equation}
	(H^{\bullet}(X,\bbC), \wedge\omega)\coloneqq [\ldots\to H^{i-1}(X,\C)\overset{\wedge\omega}{\longrightarrow}H^{i}(X,\C)\overset{\wedge\omega}{\longrightarrow}H^{i+1}(X,\C)\to\ldots]
	\label{eq:resonance}
\end{equation}
is exact for all $i<k$. This prompts the following

\begin{definition}[Holomorphic resonant varieties]\label{def:resonance}
	The $k$-th {\sl holomorphic resonant variety} of $X$ is defined as
	$$ \sR^k(X)\coloneqq\{ \omega \in H^0(X,\Omega_X^1) \mid  H^k(H^{\bullet}(X, \bbC), \wedge\omega) \neq 0\} $$
	and we set $\sR(X)=\bigcup_k \sR^k(X)$. 
	We will refer to the sequence $(H^{\bullet}(X,\bbC), \wedge\omega)$ above as the \textsl{resonance sequence}.
\end{definition}

More generally, both $\sK^{\bullet}(\omega)$ and
$\sR^k(X)$ can be twisted by unitary local systems. Recall that the space of unitary local systems is defined to be
\[\Char^0(X)^u \coloneqq \Hom(H_1(X,\bbZ)/\torsion , U(1)).\]
For any unitary character $\eta\in \Char^0(X)^u$, the corresponding local system $\mathbb{C}_{\eta}$ corresponds to a degree 0 line bundle $\sL_{\eta}\coloneqq \mathbb{C}_{\eta}\otimes_{\C}\sO_X$. In fact, this gives a one-to-one correspondence between $\Char^0(X)^u$ and $\Pic^0(X)$. 

\begin{definition}[Generalised holomorphic resonant variety]
	Given a local system $\bbC_{\eta}$ associated to a unitary character $\eta$, the $k$-th {\sl generalised holomorphic resonant variety} associated to $\eta$ is defined as  
		\[\sR^k(X, \bbC_{\eta}) \coloneqq
		\{\omega\in H^0(X, \Omega_X^1)| H^k(H^{\bullet}(X, \bbC_{\eta}), \wedge\omega) \neq 0 \}\footnote{Here we only write $\wedge\omega$ in $(H^{\bullet}(X, \bbC_{\eta}), \wedge\omega)$, since $d$ acts trivially on $H^{\bullet}(X, \bbC_{\eta})$.}.\]  
		Also, we set \[\sR_{\dol}(X) \coloneqq \bigcup_{k, \eta\  \text{unitary}}\sR^k(X,\bbC_{\eta}).\]
	
\end{definition}

As noted in the introduction, another way to understand $\sR_{\dol}(X)$ is via the tangent cone of the cohomology jump loci $\sV(X)$.  We have the following lemma due to  \cite[Proposition 3.7, Remark on p.\ 404]{GL87}, which directly generalises the so-called \textsl{tangent-cone equality}
	\be\label{eq:tangentconeeq}
	H^0(X,\Omega^1_X)\cap \TC_1\sV(X,\C_X) =\sR(X).
	\ee
	
	\begin{lemma} \label{prop:equivalence}
		With the notation as above, we have
		\[H^0(X,\Omega_X^1)\cap\TC_{\eta}(\sV^k(X,\bbC_X))=\sR^k(X, \bbC_\eta).\]
	\end{lemma}

	\br A more precise version of the above lemma  for anti-holomorphic 1-forms can be found in \cite[Theorem 1.3]{BW15} due to Budur-Wang.
\er 
The following corollary follows directly from Lemma \ref{prop:equivalence}.
\bc
Let $X$ be a smooth projective variety. Then we have
$\sT(X) = \sR_{\dol}(X).$		
\ec

\subsection{Constructible Complexes of Sheaves and Perverse Sheaves}\label{sec:perverse}

Let $D^b_c(X)$ denote the derived category of bounded constructible complex of sheaves on $X$ with coefficients in a field $\bbK$. {\sl Perverse sheaves} on  $X$ are, roughly speaking, a generalisation of local systems, i.e.\ locally constant sheaves. 
We refer the readers to \cite[Chapter 8]{HTT08} and \cite[section 5]{Dim04} for definitions and a 
comprehensive background on this topic.

Given $\sF\in D^b_c(X)$, its \textsl{characteristic cycle} $\CC(\sF)$ is a finite $\bZ$-linear combination of irreducible conic Lagrangian cycles $T^*_Z X\colon=\overline{T^*_{Z_i^{\reg}}X}$ in $T^*X$ over certain irreducible closed subvarieties $Z_i \subseteq X$
\[\CC(\sF)= \sum_i n_{Z_i} T^*_{Z_i}X.\]
Here $T^*_{Z_i^{\reg}}X$ is the conormal bundle of the regular locus  $Z_i^{\reg}$ of $Z_i$ in $X$.
For the definition when $\sF\in D^b_c(X,\bbK)$, for a field $\bbK$ in any characteristic, see e.g.\ \cite[Definition 3.34]{MS22}.

The Euler characteristic of $\sF$ satisfy  the following Kashiwara's index theorem, see \cite{Kas85} for complex coefficients and  also \cite[Theorem 3.38, Example 3.39, 3.40]{MS22} for any field coefficients.

\begin{citedthm}
	\label{thm:index} 
	For $\sF\in D^b_c(X)$  on a smooth projective variety $X$, we have
	$$\chi(X,\sF)= \CC(\sF) \cdot T^*_X X = \sum_i n_{Z_i}(T^*_{Z_i}X \cdot T^*_XX)$$
	where the dot denotes intersections of cycles in the complex manifold $T^*X$.  
\end{citedthm}

When $\sF$ happens to be a perverse sheaf $\sP$ on  $X$, $n_{Z_i}\geq 0$ (see \cite[Corollary 5.2.24]{Dim04} or \cite[Definition 3.34]{MS22} for any characteristic) and 
its singular support of $\sP $ is defined as
\[\SS(\sP) \coloneqq \bigcup_{n_{Z_i}>0}T^*_{Z_i} X.\]
Then we define the singular support of the bounded complex of constructible sheaves as follows.
\begin{definition} \label{singular support for complex}
	For a constructible complex $\sF\in D^b_c(X) $ and any integer $i$, the singular support of $\sF$ is defined as 
	\[\SS(\sF) \coloneqq \bigcup_{i}\SS(^p\sH^i(\sF)),\]
	where $ ^p\sH^i(\sF)$ is the $i$-th perverse cohomogy of $\sF$.
\end{definition}
Similar definition has also been used in \cite[Exercise X.6]{KS94} and \cite[p.\ 373]{HTT08}.

\medskip

Given $f\colon X\to A$ a morphism from s smooth projective variety $X$ to an Abelian
variety $A$, Kashiwara's estimate for the behaviour of the singular support produces a breeding ground for 1-forms with zeros. This estimate was exploited in
\cite{PS14} to show that all 1-forms admit zeros on smooth projective
varieties of general type. We recall it in our
current framework.
 Up to our knowledge, Kashiwara's estimate are only proved for complex coefficients. 

\begin{citedthm}[(Kashiwara's estimate {\cite[Theorem 4.2]{Kas76}})]
	Given $f\colon X\to A$, consider the following commutative diagram
	\begin{equation}
		\begin{tikzcd}
			T^*X \ar[dr]& f^*T^*A \ar[d]\ar[r, "f\times\id"]\ar[l, "df"'] & T^*A\ar[d]\ar[r, "\pi"]&H^0(A,\Omega_A^1)\\
			&X\ar[r, "f"] &A&
		\end{tikzcd}
		\label{eq:diagram}
	\end{equation}
	Then for any complex local system $\bbL$ on $X$
	\[\SS(\bbR f_* \bbL) \subseteq (f\times \id)(df^{-1}(0_X)),\]
	where $0_X$ denotes the zero section $T^*_XX$ of $T^*X$.
	\label{thm:kas-ss}
\end{citedthm}

Note that $ \pi(f\times \id)(df^{-1}(0_X)) = W(X)\cap H^0(A, \Omega_A^1)$ under a suitable identification. We use this result frequently as follows:
\be \label{eq:kashiwara}
\pi(\SS(\bbR f_*\bbL)) \subseteq W(X)\cap H^0(A, \Omega_A^1)
\ee

Finally we recall some special properties exhibited by the cohomology jump loci of perverse sheaves on Abelian varieties.

\begin{citedthm}[{\cite[Theorem 4.3, Corollary 1.3, Corollary 4.8]{LMW21}, \cite[Corollary 1.4]{FK00}}]\label{thm:cjlAbelian} \label{thm:lociproperties}
	Let $\sP$ be a perverse sheaf with coefficients in fields of any characteristic on a complex Abelian variety $A$ with $\dim A = g$.
	The
	cohomology jump loci of $\sP$ satisfy the following
	\begin{enumerate}
		\item\label{thm:cjlAbelian1} Propagation property:
		\[\sV^{-g}(A, \sP) \subseteq \cdots\subseteq \sV^{-1}(A, \sP) 
		\subseteq \sV^0(A, \sP) \supseteq \sV^1(A, \sP) \supseteq\cdots\supseteq \sV^g(A, \sP).\]
		Furthermore, $\sV^i(A, \sP) = \emptyset$, if $i \notin [-g, g]$.
		\item \label{thm:cjlAbelian4} Signed Euler characteristic property:
		$\chi(A, \sP) \geq 0$.
		Moreover, the equality holds if and only if $\sV^0(A, \sP) \neq \Char(A)$.
        \item For a general $\rho\in \Char^0(A)$, $H^i(A, \sP\otimes \bbL_{\rho}) = 0$ for all $i\neq 0$, where $\bbL_{\rho}$ is the rank 1 local system associated to $\rho$.
	\end{enumerate}
\end{citedthm} 
The Statement (3) in positive characteristic above is originally due to \cite[Theorem 1.1]{BSS18}.

\section{Constructible complex of sheaves on Abelian varieties}

\subsection{Linearity and comparison}\label{sec:linearity}
In this subsection, we prove that the set of 1-forms supported on the conormal sheaf of a subvariety of an Abelian variety is linear. As a consequence we obtain that the set of 1-forms associated to the singular support of any constructible complex $\sF$ is linear. 
All constructible sheaves and complexes considered in this subsection are over the complex number.

\bp \label{prop:linear}
Let $A$ be an Abelian variety. 
For any $\sF\in D^b_c(A)$,  $\pi(\SS(\sF))$ is linear in $H^0(A, \Omega_A^1)$.
\ep 
This proposition directly follows from the following

\begin{proposition}\label{van-nonsimple}
	Let $A$ be an Abelian variety and $Z$ be a proper irreducible subvariety of $A$. Then the following are equivalent
	\begin{enumerate}
		\item $Z$ is not fibred by tori and $\dim Z>0$. 
		\item General holomorphic 1-form $\omega\in H^0(A, \Omega_A^1)$ restricted to $Z^{\reg}$, i.e.\ $\omega|_{Z^{\textnormal{reg}}}$ admits  isolated zeros on the smooth locus $Z^{\textnormal{reg}}$.
	\end{enumerate} In particular, let $B\subseteq A$ be the largest (in dimensional sense) Abelian subvariety such that $Z$ is fibred by $B$,  we have $\pi(T^*_Z A) = H^0(C, \Omega_C^1)$ identified as a vector subspace  of the vector space $H^0(A, \Omega_A^1)$. Here $C$ denotes the quotient Abelian variety $A/B$.
\end{proposition}

\br When $Z$ is smooth, this result is well-known (see e.g.\ \cite[Proposition 6.3.10.]{Laz02} when $A$ is simple; it follows from \cite{PS14} when $A$ is not simple). Hacon and Kov\'acs showed this under the additional assumption that $A$ is simple \cite[Proposition 3.1]{HK05}. In fact, our proof follows from a close inspection of Hacon and Kov\'acs' argument. After this draft was written we also noticed that the result is stated in the preprint
\cite[Theorem 1]{Wei15} with a different argument. 
\er

\begin{proof}[Proof of \cref{van-nonsimple}]\label{proof:van-nonsimple} (2)$\Rightarrow$ (1): Suppose $Z$ is fibred by a Abelian subvariety $B$ and $\dim Z>0$. Let $C \coloneqq A/B$ and $Y = \varphi(Z)$  under the projection $\varphi\colon A\to C$.  
	Considering the isogeny $\tau\colon B\times C\to A$, 
	we obtain $\tau^{-1}(Z)=B\times Y$.  Then the non-trivial 1-forms coming from $B$ do not vanish on the smooth locus of $\tau^{-1}(Z)$, hence general 1-forms on $A$ do not vanish on the smooth locus of $Z$, which contradicts the assumption (2). 
	
	(1)$\Rightarrow$(2): Denote $d=\dim Z$ and $g=\dim A$. If $d=0$ it is trivial, so we assume $d>0$.  Let $N$  be the
	normal bundle of $Z^{\reg}$ in $A$. Associated to the surjection \[T_A|_{Z^{\reg}}\onto N,\]
	there is the following  chain of maps
	\[\varphi\coloneqq (\bbP(N) \to 	Z^{\reg}\times \bbP(T_0{A}) \to \bbP(T_0A)=\bbP^{g-1}).\]
It suffices to show that $\varphi$ is dominant.
	Denote by $p\colon \bbP(N)\to Z^{\reg}$ the projective bundle map. Given a point $s\in S\coloneqq \varphi(\bbP(N)),$  we can associate a hyperplane $H_s\subset T_{0}A$. Then $p$ induces an isomorphism
	\begin{equation}
		p\colon \varphi^{-1}(s)\overset{\sim}{\to} \{z\in Z^{\reg}| T_z Z^{\reg} \subset H_s\}.
		\label{eq:van-nonsimple}
	\end{equation}
	 If $\dim S<g-1$, for general $s\in S$, $Z_s \coloneqq p(\varphi^{-1}(s))$
	has dimension $g - 1 - \dim S$. Let $B$ denote the Abelian subvariety generated by $Z_s$ in $A$. Note that $B$ does not depend on general $s$, since $A$ only contains countably many Abelian subvarieties.  Also, $\dim B > g- 1 - \dim S$. Indeed, (1) implies that $Z_s$ cannot itself be an Abelian variety. By \eqref{eq:van-nonsimple}, $H_s \supset T_0 B$
	for general $s\in S$. Thus 
	$\dim T_0 B \leq g -1 - \dim S$, which gives the contradiction. Hence $\varphi$ is quasi-finite
	dominant morphism.

	For the second part, if $Z$ is not fibred by tori, $\pi(T^*_Z A)=H^0(A,\Omega_A^1)$. When $Z$ is fibred by some tori $B$, let $C$ and $\varphi\colon A\to C$ be as in the beginning of the proof.
	Let $\varphi^*\colon H^0(C,\Omega^1_C) \to H^0(A,\Omega^1_A)$ be the induced injective morphism. 
	We have $$\pi(T^*_Z A)=\varphi^*(H^0(C,\Omega_C^1)) .$$ Hence $\pi(T^*_Z A) $ is linear.  
\end{proof}

\begin{proof}[Proof of \cref{prop:linear}]
	The claim  follows from \cref{van-nonsimple} since $\SS(\sF)$ is a finite union of conormal sheaves along various subvarieties of $A$.
\end{proof}

A consequence of the proposition above is the following

\bc \label{coro:nbic0}
If $X$ admits a finite morphism $f\colon X\to A$ to its image, then $W(X)\cap H^0(A, \Omega_A^1)$ (under suitable identification induced by $f$) is linear. In particular, if the Albanese morphism is finite to its image, then $W(X)$ is linear.
\ec

\begin{proof}
	By \cite[Proposition 3.9 (2)]{Lia18} and \cite[Proposition 3.3]{Sab85}, $(f\times\id)(df^{-1}(0_X))$ in the diagram (\ref{eq:diagram}) is Lagrangian, i.e.\ it is a finite union of conormal sheaves along various subvarieties of $A$. Then the corollary follows from \cref{van-nonsimple}.
\end{proof}

\begin{proof}[Proof of \cref{thm:perverse}]  The proof is divided into 2 steps. 
	
	\medskip 
	
	\textbf{Step 1:} We first  prove the case where $\sF=\sP$ is a perverse sheaf on $A$.
	
	Note that for a short exact sequences of perverse sheaves on $A$
	$$ 0\to \sP' \to \sP \to \sP''\to 0,$$
	we have $\sV(A,\sP)=\sV(A,\sP') \cup \sV(A,\sP'')$ and  $\SS(\sP)=\SS(\sP')\cup\SS(\sP'')$. Since perverse sheaves admit Jordan--Holder type filtration with simple perverse sheaves as quotients, 
	it is enough to deal with the case of simple perverse sheaves. From now on let us assume that $\sP$ is a simple perverse sheaf on $A$. 
	First note that by the propagation property in
	\cref{thm:cjlAbelian} (\ref{thm:cjlAbelian1}), we have 
	$\sV(A, \sP) = \sV^0(A,\sP)$. According to \cref{thm:cjlAbelian} (\ref{thm:cjlAbelian4}), 
	the argument can be split in 
	two cases:
	
	\medskip
	
	\textsl{Case I: $\chi(A, \sP)>0$.} In this case \cref{thm:cjlAbelian} (\ref{thm:cjlAbelian4}) shows  that 
	$\sV^0(A,\sP)=\Char(A)$, hence $\sT(A, \sP)= H^0(A,\Omega_A^1)$.
	On the other hand by
	the Kashiwara index \cref{thm:index}
	we have
	\[\chi(A,\sP) =  \CC(\sP)\cdot T^*_AA.\]
	Note that if $Z\subset A$ is fibred by an Abelian subvariety, $( T^*_ZA\cdot T^*_AA) = 0$. 
	Therefore, $\SS(\sP)$ must contain a subvariety $Z\subset A$ such that
	$Z$ is not fibred by tori. 
	By \cref{van-nonsimple}, we conclude that $\pi(T^*_ZA) = H^0(A,\Omega_A^1)$ and the desired equality follows.
	
	\medskip
	
	\textsl{Case II: $\chi(A, \sP)=0$.} 
	As in \cite[Main Theorem and Lemma 6]{Wei16} we have
	\[\sP\otimes \C_\rho \simeq \varphi^*\sP_C [\dim A-\dim C]\]
	with notations from before.
	Since $\chi(C,\sP_C)>0$, by Kashiwara's index theorem (see also Lemma 3.6 in v3 of this article on arXiv) there exists a component $T^*_ZA\subset \SS(\sP)$ such that $\varphi(Z)\subset C$ is not fibred by tori. 
	Since $\sP$ and $\sP\otimes \C_\rho$ have the same singular support, by \cref{van-nonsimple} 
	we conclude that
	\be
	\pi(\SS(\sP)) = \varphi^* H^0(C, \Omega_C^1).
	\label{eq:ssp}
	\ee
	From Case I above we have 
	\[\sT(C, \sP_C)= H^0(C,\Omega_C^1) 
	.\]
	On the other hand, it follows from \cite[Theorem 5.5]{LMW21}  that
	\begin{equation}
		\sV^0(A, \sP)=\rho^{-1} \cdot \varphi^*(\sV^0(C, \sP_C))
		\label{eq:svp}
	\end{equation}
	where $\varphi^*\colon \Char(C) \to \Char(A)$ is given
	by the induced representation. Putting \eqref{eq:ssp} and \eqref{eq:svp} together, the desired equality follows.

	\medskip
	
	\textbf{Step 2:} In general for any $\sF\in D_c^b(A)$, \cite[Proposition 6.11]{LMW21} shows that 
	$$  \sV(A,\sF) = \bigcup_i \sV^0(A, {^p\sH^i(\sF)}),$$
	where $^p\sH^i(\sF)$ is the $i$-th perverse cohomology of $\sF$. 
	On the other hand, $\SS(\sF)= \bigcup_i \SS(^p\sH^i(\sF))$ by Definition \ref{singular support for complex}. Then the claim follows.
\end{proof}

We are now ready to prove Theorem \ref{thm:etale linearity}.
Let us first introduce a general version for any morphism $f\colon X\to A$ from a smooth projective variety $X$ to an Abelian variety $A$. 
Let

$$\sR_{\dol}(X,f)\coloneqq \bigcup_{\eta\in \Char(A)^u}\{\omega\in H^0(A,\Omega_A^1)| (H^{\bullet}(X,f^*\bbC_{\eta}), \wedge f^*\omega) \text{ not exact}\},$$

\bt \label{thm:linearity}
With the above hypothesis and notations, we have $$ \pi(\SS(\bbR f_*\C_X))= \sT(A,\bbR f_* \C_X) =\sR_{\dol}(X,f),$$
which is a finite union of vector subspaces of $H^0(A,\Omega_A^1).$
\begin{proof}
	The first equality follows from Theorem \ref{thm:perverse}. The second equality follows from \cref{prop:equivalence}. Finally the statement about linearity follows from \cref{prop:linear}.
\end{proof}

\et

\begin{proof}[Proof of Theorem \ref{thm:etale linearity}]
	Given any finite \'etale cover $\tau\colon X'\to X$, by \cref{thm:perverse} we have
	\[\pi(\SS(\bbR (a\circ\tau)_*\C_X))= \sT(A_X,\bbR (a\circ\tau)_* \C_X).\]

	Given any such finite \'etale cover $\tau\colon X'\to X$, by the lemma below there exists a finite \'etale cover $\sigma\colon\tilde{X'}\to X'$ such that 
	\be
	\sT(A_X,\bbR (a\circ\tau)_* \C_X) = \{\omega \in V| 
	(H^{\bullet}(\tilde{X'}, \bbC), \wedge (\sigma\circ\tau)^*\omega) \text{ is not exact}\} 
	\ee

	Then the claim follows by taking unions over all possible finite \'etale covers on both sides.
\end{proof}

\begin{lemma}  Consider $f\colon X\to A$ to be a morphism from a smooth projective variety $X$ to an Abelian variety $A$. Then there exists a finite \'etale Abelian cover $\sigma\colon \tilde{X}\to X$ such that 
	$$ \sT(A_X,\bbR f_* \C_X) = \{\omega \in H^0(A,\Omega^1_A)| 
	(H^{\bullet}(\tilde{X}, \bbC), \wedge (f\circ\sigma)^*\omega) \text{ is not exact}\} 
	$$
\end{lemma}

\begin{proof}
	Recall that $\sV(A,\bbR f_*\C_X)$ has finitely many irreducible components, say $\{S_1,\cdots,S_k\}$ and every irreducible component $S_i$ is a torsion translated subtori of $\Char^0(A)$ \cite{Sch15}. 
	Then there exists a finite Abelian cover   $\sigma\colon \widetilde{X}\to X$ such that $\bigoplus_{i,j} f^*\C_{\rho_i^j}=\bbR\sigma_*\C_{\widetilde{X}}$ for $i=1,\cdots, k$ and finitely many powers $j$ for each $\rho_i$. 
	
	Now, for any $\rho\in \Char^0(X)$, by the projection formula we have 
	\begin{equation} \label{projection formula}
		H^*(\widetilde{X}, \sigma^* \C_\rho)\cong H^*(A, \bbR\sigma_* \C_{\widetilde{X}}\otimes \C_\rho ).
	\end{equation}
	In particular, every component $\rho_i^{-1}\cdot S_i$ contains the constant sheaf $\C_{A}$. Furthermore, for any $j$ we have 
	$\sV(\bbR f_*\C_X\otimes \bbC_{\rho}) = \rho^{1-j}\cdot \sV(\bbR f_*\C_X\otimes \bbC_{\rho^j})$.
	Hence
	\[\sT(A, \bbR f_*\C_X)= \TC_1\sV(A, \bbR(f\circ\sigma)_* \C_{\widetilde{X}})).\]
	Then the claim follows by the tangent-cone equality (\ref{eq:tangentconeeq}).
\end{proof}

\subsection{Proof of \cref{cor general type}} \label{sec:SY22}
 So far all the results in this section are about bounded constructible complexes of sheaves with complex coefficients.
In this subsection, we prove when $A$ is a simple Abelian variety \cref{thm:perverse} holds for any algebraic closed field coefficients. As most of the definitions and tools used in the proof of \cref{thm:perverse} works over $\bbK$ (see \S \ref{sec:perverse}), the main difference here is in how we reduce the argument (see Step 2) from bounded complexes of constructible sheaves to perverse sheaves. In characteristic 0, we resorted to \cite{LMW21B} for this.

\begin{proposition}\label{any field} 
Fix any algebraically closed field $\bbK$. Let $\sF$ be a bounded constructible complex of sheaves with field coefficients $\bbK$ on a complex simple Abelian variety $A$. Then we have
\begin{itemize}
\item either $\sV(A,\sF)=\Char^0(A,\bbK)$ and $\pi(\SS(\sF))=H^0(A,\Omega^1_A)=\sT(A,\sF)$
\item or  $\sV(A,\sF)\neq \Char^0(A,\bbK)$  and $\pi(\SS(\sF))=\{0\} =\sT(A,\sF)$. 
\end{itemize} 
Here $\Char^0(A,\bbK)=\Hom(H_1(A,\Z),\bbK^*)$ is the moduli space of rank one $\bbK$-local system on $A$.
\end{proposition}
\begin{proof}
The proof is divided into 2 steps. 
	
	\medskip 
	
	\textbf{Step 1:} We first  prove the case where $\sF=\sP$ is a perverse sheaf on $A$. Note that $\sV(A,\sP)=\sV^0(A,\sP)$ due to the propagation property, see e.g. \cite[Theorem 4.7]{LMW21}.
Since $A$ is simple,  then there are two possibilities as follows 
\begin{itemize}
\item  Either $\sV^0(A,\sP)=\Char^0(A,\bbK)$ and $\sT(A, \sP)= H^0(A,\Omega_A^1)$.
	
	In this case, by \cref{thm:lociproperties}(3) we know that $\chi(A,\sP)\neq 0$. Hence, by the Kashiwara index \cref{thm:index}
	we have
	\[\chi(A,\sP) =  \CC(\sP)\cdot T^*_AA \neq 0.\] 
	Therefore, $\SS(\sP)$ must contain a subvariety $Z\subset A$ such that
	$Z$ is not fibred by tori. 
	By \cref{van-nonsimple}, we conclude that $\pi(T^*_ZA) = H^0(A,\Omega_A^1)$ and the claim follows.

\item   Or, $\sV^0(A,\sP)\neq \Char^0(A,\bbK)$ and hence $\chi(A,\sP)= 0$. By \cite[Proposition 10.1]{KW15} (the proof for simple abelian varieties works over any characteristic; alternatively see the arxiv version 1 of \cite{LMW21B}),  $\sP$  is a shifted local system. 
Since $\bbK$ is algebraically closed, any $\bbK$-local system on $A$ is a extensions of rank one local systems. Hence
 $\sV(A,\sP)$ are just finitely many points in this case and $\sT(A, \sP)= \{0\}$.
 
On the other hand, since $\sP$ is   a shifted local system we have that $\SS(\sP)= T^*_A A$, hence $\pi(T^*_A A)=\{0\}. $
\end{itemize}

\textbf{Step 2:} 
More generally for any bounded complex of constructible sheaves $\sF$ and   any rank one $\bbK$-local system $\bbL\in \Char^0(A,\bbK)$, we consider  the perverse cohomology spectral sequence
$$E_2^{i, j}=H^{i}(A, \,^p\sH^j(\sF)\otimes_\bbK \bbL)\Rightarrow H^{i+j}(A, \sF\otimes_\bbK \bbL).$$
Observe that $\bigcup_j\sV^0(A, ^p\sH^j(\sF)) \supseteq \sV(A, \sF)$.

Since $^p\sH^j(\sF)$ are perverse sheaves, with Step 1 at our disposal we proceed as follows:
\begin{itemize}
\item Either $\sV(A,\sF) = \Char^0(A,\bbK)$.
In this case, $\sV^0(A,  ^p\sH^j(\sF)) = \Char^0(A,\bbK)$ for some $j$.
 Then $\pi(\SS(^p\sH^j(\sF)))=H^0(A,\Omega^1_A)$, and hence $\pi(\SS(\sF))=H^0(A,\Omega^1_A)$ by definition.

\item Or $\sV(A,\sF) \neq \Char^0(A,\bbK)$.
In this case, we claim that $\sV^0(A,  ^p\sH^j(\sF)) \neq \Char^0(A,\bbK)$ for all $j$.  
Indeed, by choosing $\bbL\in \Char^0(A,\bbK)$ generically, by \cref{thm:lociproperties} (3) we have $ H^{i}(A, \,^p\sH^j(\sF)\otimes_\bbK \bbL)=0$ for any $i\neq 0$. So the spectral sequence degenerates at the second page. It shows that $H^j(A, \sF\otimes_\bbK \bbL)\neq 0 $ for generic $\bbL\in \Char^0(A,\bbK)$, giving $\sV(A,\sF)=\Char^0(A,\bbK)$, a contradiction.

Then by Step 1,  $\pi(\SS(^p\sH^j(\sF)))=\{0\}$, and hence $\pi(\SS(\sF))=\{0\}$ by definition.  
On the other hand, 
  we claim that in this case, in fact, one has $\sV(A,\sF) = \bigcup_j \sV^0(A,^p\sH^j(\sF))$. To see this, for any $\bbL\in \bigcup_j \sV^0(A,^p\sH^j(\sF))$, let $j'$ be the lowest degree such that $\bbL\in \sV^0(A,^p\sH^{j'}(\sF))$.  Set $\dim_\bbC A=d$. Since $ ^p\sH^j(\sF)$ is a local system, it must be that   $^p\sH^{j'}(\sF)\otimes_{\bbK} \bbL[-d]$ contains the constant sheaf as a sub-local system. In particular, $$E_2^{-d,j'}=H^{-d}(A, ^p\sH^{j'}(\sF)\otimes_{\bbK} \bbL)\neq 0.$$ Meanwhile,
 for any $i<-d$,   because of \cite[Theorem 4.7(i)]{LMW21}, or for any $j<j'$, because of the choice of $j'$ and the propagation property in \cref{thm:lociproperties} (1), we have 
 $E_2^{i,j} =0.$ 
Therefore, the above spectral sequence satisfies $$  E_\infty^{-d,j'}=E_2^{-d,j'}\neq 0.$$ Hence $H^{-d+j'}(A,\sF\otimes_\bbK \bbL)\neq 0$ and $\bbL\in \sV(A,\sF)$.
Now $\sF$ being non-zero implies that $\sV(A,\sF) $ are just finitely many points. The claim follows.
\end{itemize}
 \end{proof} 

Now we are ready to prove \cref{cor general type}.
\begin{proof}[Proof of \cref{cor general type}] 
Following the proof of \cite[Proposition 3.1]{SY22},  up to perturbing $p_X^*(d\theta)$ slightly and  multiplying
by a suitable integer, we reduce to the case where $p_X^*(d\theta)\in f^*H^1(A,\Z)$.

As observed in the proof of \cite[Lemma 3.3]{SY22}, we only need to prove that $\bbR f_*\bbK_X$ is locally constant for any algebraically closed field $\bbK$. Then it follows from Qin--Wang's result \cite[Proposition 5.4]{QW18} (see also \cite[Proposition 3.1]{SY22})  that 
$\chi(A, {^p}\sH^j \bbR f_*\bbK_X)=0$ or equivalently $\sV^0(A, {^p}\sH^j \bbR f_*\bbK_X) \neq \Char^0(A,\K) $ for any $j$.
Their proof is based on the observation that the eigenvalues  on the cohomology of the fibers induced by the monodromy action of the circle bundle is a finite set.
Since $A$ is simple, \cref{any field} 
implies that  $\pi(\SS(\bbR f_*\bbK_X))=\{0\}$ and hence ${^p}\sH^j \bbR f_*\bbK_X[-j]=R^jf_*\bbK_X$ are  local systems for all $j$. 
\end{proof}

When $A$ is not necessarily simple, assuming $\bbK = \bbC$ we prove a slightly stronger version.
\begin{corollary} \label{corgentype-nonsimple} 
 With the same assumptions and notations as in \cref{cor general type}, without the simplicity of $A$, one has 
  for every irreducible component $T^*_Z A$ of $\SS(\bbR f_* \bbC_X)$ the sub-variety $Z$ is not of general type.
\end{corollary}

\begin{proof}
Assume that there exists an irreducible component $T^*_Z A$ of $\SS( \bbR f_*\bbC_X)$, where $Z$ is of general type.  
By \cite[Theorem 3.10]{Uen73}, $Z$ is not fibred by any sub-Abelian variety of $A$. Hence $\pi(T^*_ZA) = H^0(A, \Omega_A^1)$ by \cref{van-nonsimple}. In particular, \cref{thm:etale linearity} implies that $\sV(A,\bbR f_* \C_X)=\Char^0(A,\C)$. On the other hand, the same proof as in \cref{cor general type} shows that
one can find rank one $\C$-local system $\C_\rho$ on $A$ such that $$H^*(A, \bbR f_*\C_X \otimes \C_\rho )=0$$
for all degrees. Hence $\sV(A, \bbR f_* \C_X)\neq \Char^0(A,\C)$, which gives a contradiction.

\end{proof}

\section{(Logarithmic) 1-forms with codimension one zeros}

\subsection{Arapura's result about cohomology jump loci}
Let $X$ be a smooth projective variety with a simple normal crossing divisor $D$. Set $U=X-D$.
Note that the space of logarithmic 1-forms $H^0(X, \Omega_X^1(\log D))$ does not depend on the choice of the good compactification of $U$. 
Similar to the projective case,  one can define \cite{BWY16}  $$W^i(X,D)\coloneqq\{ \omega\in H^0(X, \Omega_X^1(\log D)) \mid \codim_X Z(\omega) \leq i \}.$$
where $Z(\omega)$ is the zero locus of $\omega$. By Chevalley's upper-semicontinuity theorem, $W^i(X, D)$ are all algebraic sets.

Define \footnote{Compare this notion to $\sV^i(U,\C_U)$. In this section we will use $\Char(U)$ instead of $\Char^0(U)$. Note that when $H_1(U,\Z)$ has no torsion, $\Sigma^i(U)=\sV^i(U,\C_U).$} \[\Sigma^1(U) \coloneqq \{\rho\in\Char(U)\mid
H^1(U, \bbC_{\rho})\neq 0\}.\]
Arapura's work gives a  geometric interpretation of
the set $\Sigma^1(U)$. We briefly outline it here.  An algebraic morphism $f\colon U \to C$ from $U$ to a smooth curve $C$ is called an \textsl{orbifold map}, if $f$ is surjective, has connected generic fibre, and one of the following condition holds:
\begin{itemize}
\item $\chi(C)<0$
\item $\chi(C)=0 $  and $f$ has at least one multiple fibre.
\end{itemize} 
Roughly speaking, Arapura \cite{Ara97} (also see \cite[Corollary 5.4, Corollary 5.8]{Dim07}) showed that every positive dimensional component of $\Sigma^1(U)$ arises from some orbifold map. More precisely, an orbifold map induces an injection: $f^*\colon H^1(C, \C) \to H^1(U,\C)$.
Arapura's work implies that $$\bigcup_{\rho}   \TC_{\rho} \Sigma^1(U)=  \bigcup_f \im f^*,$$
where the first union is running over representative points from irreducible components of $\Sigma^1(U)$, 
$\TC_{\rho} \Sigma^1(U) \subseteq H^1(U, \bbC)$ denotes the tangent cone at $\rho$, and 
the second union runs over all possible orbifold maps for $U$. 
In particular, there are at most finitely many equivalent orbifold maps for a fixed $U$ \cite[Theorem 1.6]{Ara97}. So the second union is indeed a finite union.
We define $$\sT_\Sigma^1(U)\coloneqq H^0(X, \Omega_{X}^1(\log D))\cap \big(  \bigcup_{\rho}   \TC_{\rho} \Sigma^1(U)\big)= H^0(X, \Omega_{X}^1(\log D))\cap\big(  \bigcup_f \im f^* \big).$$ 
In particular, $\sT_\Sigma^1(U)$ is a finite union of vector subspaces.

\begin{remark}\label{rem:TX}
(1)  Consider the following construction given in \cite[Example 1.11]{DJL17}, which shows that $\sT^1_\Sigma(X)$ indeed captures more information than $\sT^1(X)$ in general. 
	
	Let $C_1$ be a higher genus curve that admits a degree 2 finite morphism to an elliptic curve $E$, and let $C_2$ be an elliptic curve. Consider $\sigma_1$ to be the involution action such that $C_1/\sigma_1 \simeq E$ and $\sigma_2$ induces an isogeny $C_2 \to C_2/\sigma_2$. Then their example is given by $X \coloneqq C_1\times C_2/\sim$, where the $\sim$ is a diagonal action induced by $\sigma_1$ and $\sigma_2$. In this case, one can compute that $$\sT^1(X) =\sR_{\dol}(X) = \{0\}$$
	but one can check that $W(X) = f^*H^0(E, \Omega_E^1)$ for the natural map $f\colon X\to E$ and $$\sT_\Sigma^1(X) = \sT^1(A_X,(a\circ\tau)_*\bbC_{X'}) = H^0(E, \Omega_E^1),$$ where $\tau$ is the \'etale covering $X'\coloneqq C_1\times C_2\to X$.

(2) In general, one should not expect $\sT_\Sigma^1(X) = W^1(X)$. For instance, let $X$ be a complex Abelian surface and $Y$ be the blowup of $X$ along a point. Then we take $Z$ to be the blowup of $Y$ along a point in the exceptional divisor. Then $\sT_\Sigma^1(X)=\sT_\Sigma^1(Z)$, but  $W^1(X)\subsetneq W^1(Z)$.
\end{remark}

\subsection{Projective case}

The observation of Arapura discussed above allows us to turn the piece of $W^1(X)$ that traditionally arised from cohomology jump loci into a set arising out of orbifold maps. In fact we will see in \cref{Thm:Proj-codim1} that holomorphic 1-form in $W^1(X)\setminus \sT_\Sigma^1(X)$ vanishes along some negative divisors.
See \cref{thm:codim1quasiprojective} for its quasi-projective incarnation.

\begin{definition}
Let $X$ be a smooth projective variety of dimension $n$ with a fixed ample line bundle $H$ and an integral divisor $E$ on $X$. We say that $E$ is $H$-negative if $E^2\cdot H^{n-2}<0$, where $$E^2\cdot H^{n-2}=\int_X c_1^2(E)\wedge c_1^{n-2}(H).$$ Similarly, $E$ is called $H$-trivial if $E^2\cdot H^{n-2}=0$. 
\end{definition}  
 
We need the following a more precise version of \cite[Theorem 2]{Spu88}.
\begin{lemma}\label{lem:welldef}
Let $E$ be an integral divisor in $X$. Suppose there exists a holomorphic 1-form $\omega$ such that $E \subseteq Z(\omega)$. Then the following statements are true:

$(1)$ 
When $E$ is $H$-trivial, there exists an orbifold map $f\colon X\to C$ such that $\omega=f^*\eta$ for some $\eta\in H^0(C, \Omega_C^1)$ and $E$ is the only component of the fibre of $f$ containing $E$.

$(2)$ the sign of the intersection number $E^2\cdot H^{n-2}$does not depend on the choice $H$.
\end{lemma}

\begin{proof}
For (1), by \cite[Theorem 2]{Spu88}, we only need to show that $E$ is the unique component in the fibre containing it. To this end, let $E'$ be the union of components not supported on $E$ in the scheme-theoretic fibre containing $E$ and let $aE+E'$ denote the fibre class for some positive integer $a$. 
Since $f$ has connected fibres, $E\cdot E'\cdot H^{n-2}>0$.   
On the other hand since $E'$ is contained in a fibre, we have $(aE+E')\cdot E'\cdot H^{n-2}=0$. Then from $(aE+E')^2\cdot H^{n-2}=0$, we get $E^2\cdot H^{n-2}<0$, which is a contradiction.

To see (2), note that if for any ample class $H$, $E$ is $H$-nonnegative, i.e., $E^2\cdot H^{n-2}\geq0$, by \cite[Theorem 2]{Spu88} it must be $H$-trivial and then by (1), we know that $E$ is the unique component of the fibre of $f$. Therefore for any other ample class $H'$, $E$ must be $H'$-trivial. As a consequence, if $E$ is $H$-negative, it is $H'$-negative for any other ample class $H'$.
\end{proof}

\noindent We denote $$W_{\textnormal{neg}}(X)\coloneqq\{\omega\in H^0(X, \Omega_X^1)\ |\exists \textnormal{ some  negative  integral divisor}\  E\  \textnormal{ such that } E\subset Z(\omega)\}.$$ 
Then we have the following result.

\begin{theorem}\label{Thm:Proj-codim1}
 Let $X$ be a smooth projective variety of dimenison $n$. With the above notations, we have  
$$W^1(X)= \sT_\Sigma^1(X)\cup W_{\textnormal{neg}}(X).
$$
In particular, $W^1(X)$ is  linear. 
\end{theorem}

\begin{remark} (1) \cref{Thm:Proj-codim1} complements the result of Green-Lazarsfeld \cite{GL87} which ensures the linearity of $\sT_\Sigma^1(X) \subset W^1(X)$. As noted in \cref{rem:TX} (2) this is often a proper subset.

(2) The two pieces $\sT_\Sigma^1(X)$ and $W_{\textnormal{neg}}(X)$ may overlap. For example let $f\colon S\to C$ be a morphism from a smooth  projective surface $S$ to a smooth projective curve $C$ with genus $g(C)\geq2$. Take a 1-form $\omega\in H^0(C, \Omega_C^1)$ which has a zero at $p\in C$. Let $X$  be the blow-up of $S$ along a point in $f^{-1}(p)$. The exceptional curve $E$ has negative self-intersection. Consider the natural morphism $f'\colon X\to C$, then  $(f')^*\omega\in \sT_\Sigma^1(X)\cap W_{\textnormal{neg}}(X)$.
\end{remark}

\begin{lemma} \label{countable} Let $X$ be a  smooth  projective variety of dimension $n$ with an ample divisor $H$. Then there are at most countably many $H$-negative divisors in $X$.
\end{lemma}

\begin{proof}
Let $E$ be any $H$-negative divisor. Let $S$ be a general complete intersection surface by the hyperplanes in $|mH|$ for $m\gg 0$. Then $E\cap S$ is a negative curve.
Since there are at most countably many negative curves on $S$,  the claim follows.
\end{proof}

\begin{proof}[Proof  of \cref{Thm:Proj-codim1}]
We assume  $n>1$. For any 1-form $\omega\in W^1(X)$, there is an integral divisor $E\subset X$ such that $E \subseteq Z(\omega)$. By \cite[Theorem 2]{Spu88} we have either $\omega\in W_{\textnormal{neg}}$, or there exists an orbifold map $f\colon X\to C$ with genus  $g(C)>0$  and $\omega=f^*\eta$ for some $\eta\in H^0(C, \Omega_C^1)$. In the latter case by \cref{lem:welldef} we know that $E$ is the only component of a fibre. Since $E\subseteq Z(f^*\eta)$ and $E$ is the whole fibre, either $f$ has a multiple fibre and  $g(C)=1$, or $g(C)>1$ and $\eta(f(E))=0$. Hence  $\omega\in \sT_\Sigma^1(X)$. Notice that $\sT_\Sigma^1(X)\subseteq W^1(X)$ and hence  the first part of the theorem follows.

For the second part, note that
$W^1(X)$ is an algebraic set. 
Since $\sT_\Sigma^1(X)$ is  linear and $W_{\textnormal{neg}}$ is a union of at most countably many linear subspaces in $H^0(X, \Omega_X^1)$ by \cref{countable} and \cref{lem:welldef}, $W^1(X)$ is also linear.
 \end{proof}

 \subsection{Quasi-projective case}

Let $X$ be a smooth projective variety with a simple normal crossing divisor $D=\sum_{j=1}^r D_j$. Set $U=X-D$.
Similar to the projective case, we denote
$$W_{\textnormal{neg}}(X, D)\coloneqq\{\omega\in H^0(X, \Omega_X^1(\log D))\ |\ E\subseteq Z(\omega)   \ \textnormal{for some negative integral divisor}\  E\}.$$ 
 
\begin{theorem} \label{thm:codim1quasiprojective} 
With the above notations, we have
$$
W^1(X,D)= \sT_\Sigma^1(U) \cup W_{\textnormal{neg}}(X, D).
$$
In particular $W^1(X,D)$ is linear in $H^0(X, \Omega_X^1(\log D))$. 
\end{theorem}

The proof of Theorem \ref{thm:codim1quasiprojective}  follows that of Theorem \ref{Thm:Proj-codim1} closely with Theorem \ref{thm:spurr} in appendix, which is a generalisation of Spurr \cite[Theorem 2]{Spu88} for pairs.

\br\label{rem:bwy}  In \cite{BWY16}, Budur, Wang and Yoon identified a linear piece of $W^1(X,D)$; namely  
$$ \left(\textbf{R}^1 \cup \textbf{R}_{2n-1}\right) \cap H^0(X,\Omega^1_X(\log D)) \subseteq W^1(X,D).$$
Here we use the same notations as in their paper. 
Note that  $\textbf{R}^1 \cap H^0(X,\Omega^1_X(\log D))$coincides with $\sT^1_\Sigma(U)$. 
But it is not clear to us how $\textbf{R}_{2n-1} \cap H^0(X,\Omega^1_X(\log D))$ is connected to $W_{\textnormal{neg}}(X, D)$.

 Dimca in \cite{Dim10} define the first logarithmic resonance variety 
\[\mathcal{LR}_1(U) \coloneqq \{\omega\in H^0(X,\Omega_X^1(\log D)) \mid H^1(H^0(X, \Omega_X^{\bullet}(\log D)), \wedge\omega)\neq 0\}\]  as the same notation in his paper. In particular,  \cite[Proposition 4.5]{Dim10} implies that $\mathcal{LR}_1(U)=\bigcup_f \im f^*$, where the union runs over all possible orbifold maps $f:U \to C$ with $\chi(C)<0$ and $C$ not being a once-punctured elliptic curve. Hence $\mathcal{LR}_1(U)\subseteq \sT^1_\Sigma(U)$.
\er

\appendix

\section{Logarithmic Generalization of A Theorem of Spurr}

We show the following generalization of Spurr \cite[Theorem 2]{Spu88} for pairs. 

\begin{theorem} \label{thm:spurr}
Let $(X, D)$ be a pair with $X$ a smooth projective variety of dimension $n$ and $D$ a simple normal crossing divisor on $X$. Let $H$ be an ample divisor on $X$. If $(X, D)$ carries a nonzero logarithmic 1-form $\omega\in H^0(X, \Omega_X^1(\log D))$  such that there exists an integral divisor $E$ with $E^2\cdot H^{n-2}\geq0$, $E\subseteq Z(\omega)$, and $E\not\subseteq D$, then there is a surjective morphism $f\colon X-D\to C$ to a smooth quasi-projective curve $C=\bar{C}-B$ (where $\bar{C}$ is a smooth completion of $C$ and $B$ can be empty) with 

(1)  $\chi(C)\leq 0$ and $\omega=f^*\eta$ for some $\eta\in H^0(\bar{C}, \Omega_{\bar{C}}^1(\log B))$.

(2) $f$ has  connected generic fibres.

(3) $E^2\cdot H^{n-2}=0$.

(4) If $\chi(C)=0$, then $f$ has at least one multiple fibre.
\end{theorem}

To prepare for the proof, recall the following
construction of Albanese varieties, which can be found in \cite{Iit76} (see also \cite{Fuj15} for a survey).

Let $X$ be a smooth projective variety of dimension with a simple normal crossing divisor $D$. Pick a basis $\{\theta_1,\ldots\theta_q\}$ for $H^0(X, \Omega_X^1)$ and $\{\omega_1, \ldots, \omega_r\}\in H^0(X, \Omega_X^1(\log D))$ such that $\{\theta_1, \ldots, \theta_q, \omega_1, \ldots, \omega_r\}$ is a basis of $H^0(X, \Omega_X^1(\log D))$. Pick a basis $\{\gamma_1, \ldots, \gamma_{2q}\}$ for the free part of $H_1(X, \mathbb{Z})$ and a basis $\{\delta_1, \ldots, \delta_r\}$ for the free part of $ \ker\{ H_1(U, \mathbb{Z})\to H_1(X, \mathbb{Z})\}.$ Then we have the following periods as a semi-lattice for $H^0(X,\Omega_X^1(\log D))^\vee$ 
\[\Lambda=\sum_{i=1}^{2q}\mathbb{Z}\left(\int_{\gamma_i}\theta_1, \ldots, \int_{\gamma_i}\omega_r\right)+\sum_{j=1}^r\mathbb{Z}\left(\int_{\delta_j}\theta_1, \ldots, \int_{\delta_j}\omega_r\right).\]
The Albanese variety is then defined as the semi-Abelian variety $A_U =\displaystyle \frac{H^0(X,\Omega_X^1(\log D))^{\vee}}{\Lambda}$ and the Albanese map $a_U\colon U \to A_U$ is given by 
\[a_U(x)=\left[\sum_{i=1}^q\left(\int_{p}^x\theta_i\right)\theta^*_i+\sum_{j=1}^r\left(\int_{p}^x\omega_j\right)\omega_j^*\right]\slash\Lambda,\]
 where $p\in U$ is a chosen base-point and $\theta_i^*, \omega_j^*$ are the dual bases in $H^0(X, \Omega_X^1(\log D))^{\vee}$.

\begin{proof}[Proof of \cref{thm:spurr}]
First we claim that if $E\cap D_i \neq \emptyset$ for some component
$D_i$ of $D$, then $\omega$ has no pole along $D_i$. In other words, $\omega\in H^0(X, \Omega_X^1(\log D-D_i))$. To see
the claim consider the following diagram  
\begin{equation}\label{holo-log-form-compare}
\begin{tikzpicture}[baseline= (a).base]
\node[scale=.93] (a) at (0,0){
\begin{tikzcd}
 H^0(X, \Omega_X^1(\log D-D_i)\otimes\mathcal{O}_X(-E)) \ar[r]\ar[d]& H^0(X, \Omega_X^1(\log D-D_i))\ar[r]\ar[d] & H^0(\Omega_X^1(\log D-D_i)|_{E}) \ar[d] \\
   H^0(X, \Omega_X^1(\log D)\otimes\mathcal{O}_X(-E))\ar[r]& H^0(X, \Omega_X^1(\log D)) \ar[r]& H^0(X, \Omega_X^1(\log D)|_{E})
\end{tikzcd}
};
\end{tikzpicture}
\end{equation}
The left vertical arrow is an isomorphism. Indeed, it is injective and the cokernel is contained in $H^0(D_i, \sO_{D_i}(-E))$ \cite[2.3 properties]{EV92} which is zero since $E|_{D_i}$ is an effective divisor. Therefore any $\omega\in H^0(X, \Omega_X^1(\log D))$ such that $E\subseteq Z(\omega)$ must come from $H^0(X, \Omega_X^1(\log D-D_i))$. Hence the claim. 

Now it suffices to deal with the case when $D\cap E = \emptyset$. Indeed, let $D = D' + D''$ such that $E$ intersects each component of $D'$ and $E\cap D'' = \emptyset$, by the reduction step we can construct an orbifold map from $f\colon X\setminus D'' \to C$ satisfying the desired properties. Then the restriction $f|_{U}\colon U\to C$ is also an orbifold map and satisfies the same properties. In what follows we assume $D\cap E =\emptyset$.

We may assume that $\omega$ is not everywhere holomorphic, otherwise we are done by \cite[Theorem 1]{Spu88}.
 Let $\phi\colon N\to E$ be the normalisation map and $\varphi\colon N\to U$ be the composition map (this makes sense since $E\cap D=\emptyset$). Consider the  following commutative diagram 
$$
\xymatrix{
 N\ar[r]^{\varphi}\ar[d]^{a_N}& U \ar[d]^{a_U} \\
  A_N\ar[r]^{\psi} & A_U,}
$$ where $a_N$ and $a_U$ are Albanese maps and  the base points are chosen in an appropriate way such that $\psi$ is a group homomorphism.

Now we consider the transpose of the pullback map 
\[\varphi^{\vee}\colon H^0(N, \Omega_N^1)^{\vee}\to H^0(X, \Omega_X^1(\log D))^{\vee},\] which induces the morphism  $\psi$. Without any loss of generality we assume $\omega_1=\omega$ with the notations introduced shortly before the proof. 
Hence by  hypothesis $\varphi^{*}(\omega_1)=0$. Let $z_1$ be the coefficient coordinate of $\omega_1^*\in H^0(X, \Omega_X^1(\log D))^{\vee}$. We get
 $\psi(A_N)$ is contained in $\{z_1=0\}/\Lambda$.    
We define $\beta\colon U \to T\coloneqq A_U/\psi(A_N)$ as the  composition of $a_U$ and the quotient $A_U\to T$. 
Since $\alpha_U(E)\subset \psi(A_N)$, $E$ is contracted by $\beta$.

Notice that $\beta$ is not the constant map. 
We claim that $\dim\beta(U)=1$.  Otherwise, replacing $U$ by an intersection of general hyperplane sections coming from the very ample linear system $|mH|$ for some $m\gg 0$ we may assume $\dim U = \dim \beta(U) = 2.$ In this case $\beta\colon U\to \beta(U)$ is a generically finite surjective morphism. Projectivising  
and resolving indeterminacy we get a generically finite morphism $\overline{\beta}\colon\overline{U}\to\overline{\beta(U)}$ where $\overline{U}$ is smooth and projective. 
Note that  $E\cap D=\emptyset$. Hence $E^2<0$ in $\overline{U}$ (See e.g., \cite[Theorem 10.1]{KK13}),  
which contradicts the assumption $E^2\geq 0$. 

Taking the Stein factorisation of $\bar{\beta}$, we get the following commutative diagram:
$$\xymatrix{
\overline{U} \ar[rd]_{\overline{\beta}} \ar[r]^f & \overline{C} \ar[d]  \\
& \overline{\beta(U)},
}$$
 where $\overline{C}$ is the smooth curve defined by the Stein factorization and $C\coloneqq f(U)$. Then we have the following commutative diagram:
 $$ \xymatrix{
U \ar[r]^{f|_U} \ar[d] & C \ar[r] \ar[d] & \beta(U) \ar@{^{(}->}[d] \\ 
A_U \ar[r]^{\psi_f} & A_C \ar[r] & T }
 $$
Note that the holomorphic 1-form $dz_1$ on $T$ pulls back to logarithmic 1-form $\omega$. Therefore to see (1), it suffices to show that $T$ is isogenous to $A_C$. 
First all the horizontal maps in the diagram are surjective. In fact we only need to show $\psi_f$ is surjective. 
Since $f|_U\colon U\to C$ is surjective and has connected generic fibres,  the induced map on the first homology groups $H_1(U,\Z) \to H_1(C,\Z)$ is surjective and the surjectivity of $\psi_f$ follows. 
Since $f$ comes from the Stein factorisation and  $\beta$ contracts $E$,  so does $f$. 
By choosing appropriate base points, we get that $\psi(A_N)$
 is contained in the kernel of $\psi_f$. Hence $A_C$ is isogenous to $T$. To see (3), notice that $E$ is contained in a fibre of $f$ and $E^2\cdot H^{n-2}\geq 0$. Hence $E^2\cdot H^{n-2}=0$.
 Finally for (4), when $\chi(C)=0$, notice that for any non-zero $\eta \in H^0(\overline{C}, \Omega^1_{\overline{C}}(\log B))$, it has no zeros. Then $E$ has to be a multiple fibre. 
\end{proof}

\begin{proof}[Proof of \cref{thm:codim1quasiprojective}]
We assume  $n>1$. For any 1-form $\omega\in W^1(X,D)$, pick an integral divisor $E\subset X$ such that $E \subseteq Z(\omega)$.  
We may assume that $\omega \notin W_{\textnormal{neg}}(X,D)$ so that $ E^2\cdot H^{n-2}\geq 0$. If $E $ is not a  component of $D $, then $\omega \in \sT^1_\Sigma(U)$ by Theorem \ref{thm:spurr}.    Otherwise,  say $E=D_1$ a component of $D$. Then we have an injective map \cite[2.3 Properties (c)]{EV92}
\[
 H^0(X, \Omega_X^1(\log D)(-D_1))\into H^0(X, \Omega_X^1(\log D-D_1)).
\]
 In particular,  $ \omega \in H^0(X, \Omega^1_X (\log D-D_1))$. Set $U_1=X-\bigcup_{j\neq 1} D_j$.  By Theorem \ref{thm:spurr}, we have an orbifold map $f_1\colon U_1 \to C_1$ such that $\omega\in f_1^*  H^0(\overline{C}_1, \Omega^1_{\overline{C}_1} (\log B_1 )) $, where $B_1= \overline{C}_1-C_1 $.
Note that $f\coloneqq f_1 \vert_{U}\colon U \to C$ is also an orbifold map, where $C$  is the image of $U$. 
Furthermore we know that $f_1^* H^0(\overline{C}_1, \Omega^1_{\overline{C}_1} (\log B_1) )$ is contained in $f^* H^0(\overline{C}, \Omega^1_{\overline{C}} (\log B))$, where $ \overline{C}=\overline{C}_1$ and $B=  \overline{C}-C$. 
It implies that $\omega \in f^* H^0(\overline{C}, \Omega^1_{\overline{C}} (\log B))$, i.e. $\omega \in \sT^1_\Sigma (U)$. The first part follows.

To see the linearity, we notice that there are at most countably many negative divisors. Indeed, similar to the projective case the sign of the intersection of $E^2\cdot H^{n-2}$ does not depend on the choice of $H$ when $E\subset Z(\omega)$ for some $\omega\in H^0(X,\Omega_X^1(\log D))$. This can be seen using the same argument as in the proof of \cref{lem:welldef}.   
Then the proof follows the arguments in Theorem \ref{Thm:Proj-codim1} verbatim.
\end{proof}

\subsection*{Acknowledgements}
We are very grateful to the referee for carefully reading our paper and providing us with insightful input that greatly improved the exposition quality. The referee inspired us to generalize our main result to arbitrary characteristic in case of simple abelian varieties and relate it to Schreieder and Yang's work \cite{SY22}; we are thankful for their encouragement.
We also thank Nero Budur, Daniel Huybrechts, Sandor Kov\'acs, Laurentiu Maxim, Mihnea Popa, Claude Sabbah, Christian Schnell, Stefan Schreieder, Vivek Shende, Botong Wang, Lei Wu and Ruijie Yang for several insightful discussions. 
YD was supported by the Hausdorff Center of Mathematics, Bonn under Germany's Excellence Strategy (DFG) - EXC-2047/1 - 390685813 during the preparation of this manuscript. 
FH is supported by grant 1280421N from the Research Foundation Flanders (FWO). Also, part of the work was done when the second named author was supported by grant G097819N of Nero Budur from the Research Foundation Flanders (FWO).
YL is partially supported  by National Key Research and Development Project SQ2020YFA070080, the starting grant from University of Science and Technology of China, NSFC grant No. 12001511, the Project of Stable Support for Youth Team in Basic Research Field, CAS (YSBR-001),  the project ``Analysis and Geometry on Bundles" of Ministry of Science and Technology of the People's Republic of China and  Fundamental Research Funds for the Central Universities.

\end{document}